\documentclass[12pt]{amsart}

\usepackage{amssymb,latexsym}
\usepackage{amsmath,amsfonts,amsthm,amscd,amsxtra,enumerate,float}
\usepackage{caption}
\usepackage{mathtools}
\usepackage[all]{xy}
\usepackage{amsmath,latexsym,amssymb, enumerate, amsthm, epsfig, hyperref,tikz} 
\usepackage[margin=1in]{geometry}
\usepackage{epsfig}

\usepackage[utf8]{inputenc}

\newcommand{\sk}{{\ensuremath{\sf k }}}

\DeclareMathOperator{\del}{del}
\DeclareMathOperator{\lk}{lk}

\DeclareMathOperator{\Tor}{Tor}

\newtheorem{conjecture}{ Conjecture}[section]
\newtheorem{theorem}[conjecture]{ Theorem}
\newtheorem{lemma}[conjecture]{ Lemma}
\newtheorem{corollary}[conjecture]{ Corollary}

\theoremstyle{definition}
\newtheorem{definition}[conjecture]{ Definition}

\newtheorem{example}[conjecture]{ Example}
\newtheorem{remark}[conjecture]{ Remark}

\providecommand{\Soder}{\ensuremath{\text{S}\ddot{\text{o}}\text{derberg}} }

\renewcommand\dim{\text{\rm dim}}

\providecommand\Tor{{\rm Tor}}

\providecommand\reg{{\rm reg}}

\begin{document}
\title{Herzog, Hibi and Ohsugi conjecture for trees}

\author[A. Kumar]{Ajay Kumar}


\email{ajay.kumar@iitjammu.ac.in}

\author[R. Kumar]{Rajiv Kumar}

\email{gargrajiv00@gmail.com}
\address{Indian Institute of Technology Jammu, India.}

\date{\today}



\maketitle
\begin{abstract}
	Let $S=\mathbb{K}[x_1,\dots, x_n]$ be a polynomial ring, where $\mathbb{K}$ is a field, and $G$ be a simple graph on $n$ vertices. Let $J(G)\subset S$ be the vertex cover ideal of $G$. Herzog, Hibi and Ohsugi have conjectured that all powers of vertex cover ideals of chordal graph are componentwise linear. Here we establish the conjecture for the special case of trees. 
	We  also show that if $G$ is a unicyclic vertex decomposable graph that does not contain $C_3$ or $C_5$, then symbolic powers of $J(G)$ are componentwise linear. 
\end{abstract}
\section{Introduction}
Let $\mathbb{K}$ be a field and $S=\mathbb{K}[x_1,\dots, x_n]$ be a polynomial ring, $n \in \mathbb{N}_{>0}$, where $\mathbb{N}_{>0}$ denotes the set of positive integers. Set $\mathbb{N}=\mathbb{N}_{>0} \cup \{0\}$. Let $G$ be a simple graph with vertex set $V(G)=\{x_1,\ldots,x_n \}$ and edge set $E(G)=\{\{x_i,x_j\}:x_i,x_j \in V(G)  \}$. Then one can associate an \emph {edge ideal} $I(G) \subset S$ to $G$ generated by all monomials $x_ix_j$ such that $\{x_i,x_j \} \in E(G)$. The Alexander dual of $I(G)$, i.e., 
$J(G)=I(G)^{\vee}=\bigcap\limits_{\{x_i,x_j \} \in E(G)} \langle x_i,x_j\rangle$, is called the \emph{vertex cover ideal} of $G$. A graph $G$ is said to be vertex decomposable/shellable if its independence complex $\Delta(G)$ has this property. A graph $G$ is called (sequentially) Cohen-Macaulay, if the quotient ring $S/I(G)$ is (sequentially) Cohen-Macaulay. For a graph $G$, the following implications are known:
$$ \textrm {vertex decomposable} \implies \textrm{shellable} \implies \textrm{ sequentially Cohen-Macaulay}.$$ 

Eagon and Reiner \cite{Eagon} showed that a graph is Cohen-Macaulay if and only if its vertex cover ideal has a linear resolution. More generally in \cite{HHibi}, Herzog and Hibi proved that a graph is sequentially Cohen-Macaulay if and only if its vertex cover ideal is componentwise linear (see Definition \ref{component}). R\"{o}mer  \cite{Romer} observed that if Char(${\mathbb{K}}$)$=0$, then the multiplicity Conjecture due to Herzog, Huneke and Srinivasan holds for componentwise linear ideals. In 2009, this conjecture was solved by Boij-\Soder \cite{BS} and Eisenbud-Schreyer \cite{ESF}. In \cite{HRW99}, authors have proved that componentwise linear ideals are Golod.  Thus, one would like to find some classes of ideals having componentwise linear resolution. In particular, one may be interested in finding some combinatorial conditions on certain combinatorial objects (simplicial complex, graph) such that the corresponding associated ideals have (componentwise linear) linear resolution. Authors in \cite{fv} proved that the vertex cover ideal of a chordal graph is always componentwise linear. In \cite{HHO2011}, Herzog, Hibi and Ohsugi studied powers of vertex cover ideals of graphs and and proposed the following conjecture.
\begin{conjecture}\label{conjecture}
	Let $G$ be a chordal graph. Then all powers of the vertex cover ideal of $G$ are componentwise linear.
\end{conjecture}
There has been  very little progress made on this conjecture except for very few classes like  generalized star graphs, Cohen-Macaulay chordal graphs (see \cite{HHO2011,Mohammadi}). Authors in \cite{EQ2019} show that the second power of the vertex cover ideal of a path is componentwise linear. They also ask whether powers of the vertex cover ideal of a chordal graph have linear quotients or not? This question is a stronger version of the conjecture stated above. For some particular cases the above conjecture has been studied  by various authors (see \cite{Er2,Mohammadi,FM,Selva2019}).   

Authors in \cite{GRV2005} show that for a bipartite graph  $J(G)^{(k)}=J(G)^k$, where  $J(G)^{(k)}$ denotes the $k$th symbolic power of $J(G)$. Thus, to study Conjecture \ref{conjecture} for trees, one can consider symbolic powers of the associated vertex cover ideal. 
For a given graph $G$, Fakhari \cite{Fakhari} introduced a new graph $G({\bf k})$, and showed that  the polarization of $J(G)^{(k)}$ is the vertex cover ideal of $G({\bf k})$.
He also describes a relationship between algebraic properties, e.g., Cohen-Macaulayness, very well covered, of  graphs $G$ and $G({\bf k}).$ 
As a consequence, he observed that if $G$ is a Cohen-Macaulay and very well covered graph, then symbolic powers of the vertex cover ideal of $G$ have linear quotients, and hence are componentwise linear. 

The problem of finding the regularity of edge ideals and vertex cover ideals has been extensively studied since the last decade. For a graded ideal $I$ of a ring $S$, it is well known that $\reg(I^s)$ is a linear function of $s$ for  $s\gg 0$, i.e. there exist non-negative integers $a,b$ and $s_0$ such that $\reg(I^s)=as+b$ for all $s \geq s_0$ (see \cite{CHT, vijay}). Although the constant $a$ is given by the maximum degree of minimal generators of $I$, no explicit formula for $b$ and $s_0$ is known. The problem of computing the bounds for the regularity of (symbolic) powers of the vertex cover ideal of a graph has been studied by many researchers (see \cite{Er2, KKSS, Selva2019, Fakhari,  Fakhari2019}).

In this article, our main focus is to address Conjecture \ref{conjecture} for trees. For this, we study the vertex decomposable property of the graph $G({\bf k}).$ In Example \ref{example3}, we see that there exists a graph $G$ for which $G({\bf k})$ is not a vertex decomposable graph. 
For a given graph $G$, we introduce a new construction $G(\mathbf{k}_t)$ which  generalizes the construction of $G({\bf k})$, and this helps us to understand the vertex decomposability of $G({\bf k})$. Further, we prove that $G(\mathbf{k}_t)$ is vertex decomposable when $G$ is a tree. Hence we solve Conjecture \ref{conjecture} for trees. For a given vertex decomposable graph $G$, we observe that $G(\mathbf{k}_t)$ need not be vertex decomposable (see Examples \ref{example} and \ref{example2}).

We now give a brief overview of this paper. In Section \ref{section2}, we introduce basic notions of graph theory and commutative algebra. In Section \ref{sec3}, we settle Conjecture \ref{conjecture} for trees, which is a main result of this article (see Theorem \ref{tree}). 

In Section \ref{section4}, we prove that for a vertex decomposable unicyclic graph $G$ with cycle $C_n$, where $n \neq 3,5$, $G({\bf k})$ is vertex decomposable (see Theorem \ref{unicyclic}). Further, we know that the regularity of a componentwise linear ideal can be determined by the maximum degree of its minimal generators. As a consequence, we find the regularity of symbolic powers of vertex cover ideals of some classes of vertex decomposable graphs. 
\section{Preliminaries}\label{section2}

\begin{definition}
	A \emph{simplicial complex} $\Delta$ on the vertex set $V$ is a collection of subsets of $V$ which satisfy the following:
	\begin{enumerate}[i)]
		\item For $x\in V$, $\{x\}\in \Delta$.
		\item If $F\in \Delta$ and $F'\subset F$, then $F'\in \Delta$.
	\end{enumerate}
	An element of $\Delta$ is called a \emph{face} of $\Delta$ and a maximal face of $\Delta$ with respect to inclusion is called a \emph{facet} of $\Delta$.
\end{definition}
\begin{definition} Let $\Delta$ be a simplicial complex on the vertex set $V$.
	\begin{enumerate}[a)]
		\item For a face $F$ of $\Delta$, the \emph{deletion} of $F$, denoted as  $\del_{\Delta}(F)$, is a simplicial complex is defined as
		$$\del_{\Delta}(F)=\{H\in \Delta: H\cap F=\phi\}.$$
		\item Let $F\in \Delta$. Then the \emph{link} of $F$, denoted as $\lk_{\Delta}(F)$, is a simplicial complex is defined as
		$$\lk_{\Delta}(F)=\{H\in \Delta: H\cup F\in \Delta,  H\cap F=\phi\}.$$
		\item A simplicial complex $\Delta$ is said to be \emph{vertex decomposable} if it is either a simplex or else has some vertex $x$ such that
		\begin{enumerate}[i)]
			\item $\del_{\Delta}(x)$ and $\lk_{\Delta}(x)$ are vertex decomposable, and
			\item  no face of $\lk_{\Delta}(x)$ is a facet of $\del_{\Delta}(\{x\})$.
		\end{enumerate}
		A vertex $x$ which satisfies Condition (ii) is called a \emph{shedding vertex}.
		\item A simplicial complex $\Delta$ is called \emph{shellable} if there exists a linear order $F_1,\ldots,F_r$ of all facets of $\Delta$ such that for all $1 \leq i<j \leq r$, there exist $x \in F_j \setminus F_i$ and $s \in \{1,\ldots,j-1 \}$ with $F_j \setminus F_s=\{x\}$. 
	\end{enumerate}
	
\end{definition}
 \begin{definition} Let $G$ be a simple graph with the vertex set $V(G)$ and the edge set $E(G)$.
\begin{enumerate}[i)]
\item A subset $A \subset V(G)$ is called a \emph{vertex cover} of $G$, if $A \cap \{x,y\} \neq \emptyset$ for any $\{x,y \} \in E(G)$ and it is called \emph{minimal} if for any $a \in A$, $A \setminus a$ is not a vertex cover of $G$.
\item A subset $C$ of $V(G)$ is called an \emph{independence set} of $G$ if $\{x,y\} \notin E(G)$ for any $x,y \in C$. The collection  $\Delta(G)$ of all independent sets of $G$ is a simplicial complex on the vertex set $V(G)$, called the \emph{independent complex} of $G$.
\item Let $x \in V(G)$. Then an edge $\{x,y \}$ obtained by adding a new vertex $y$ at $x$ is called a \emph{whisker} of $G$.
\item Let $K \subset V(G)$. Then by $G \setminus K$, we mean the induced subgraph of $G$ on $V(G) \setminus K$.
\end{enumerate}
\end{definition}
 For a vertex $x \in V(G)$, the open neighborhood of $x$ in $G$ is defined as $N_G(x)=\{y \in G:\{x,y \} \in E(G) \}$, and   $N_G[x]=N_G(x)\cup \{x\}$ is called the closed neighborhood of $x$ in $G$. A graph $G$ is said to be a \emph{vertex decomposable} (resp. \emph{shellable})  graph if the independent complex $\Delta(G)$ is  vertex decomposable (resp. shellable). Thus the definition of a vertex decomposable simplicial complex translates to a vertex decomposable graph (see \cite{Wood2009}) as following.
 \begin{definition} A graph $G$ is said to be \emph{vertex decomposable} if it has no edges or there is a vertex $x$ in $G$ such that 
 \begin{enumerate}[i)]
 \item $G \setminus \{ x\}$ and $G \setminus N_G[x]$ are vertex decomposable, and
 \item for every independent set $C$ in $G \setminus N_G[x]$, there exists some $y \in N_G(x)$ such that $C \cup \{y\}$ is independent in $G \setminus \{x\}$.
 
\end{enumerate}  
 \end{definition}
 \begin{definition}\label{puredef}
	{\rm Let $M$ be a finitely generated $\mathbb{Z}$-graded $S$-module.
		\begin{enumerate}[i)] 
			\item Then $\beta^S_{i,j}(M) = (\dim_\sk(\Tor_i^S(M,\sk))_j$ is called the $(i,j)^{th}$ \emph{graded Betti number} of $M$. 
			\item The \emph{regularity} of $M$, denoted as $\reg(M)$, is defined as $$\reg(M)=\max\{j-i:\beta^S_{i,j}(M)\neq 0\}.$$
			\item A module $M$ is said to have a \emph{linear resolution}, if for some integer $d$,  $\beta_{i,i+b}=0$ for all $i$ and every $b \neq d$.
			\item A module $M$ is called \emph{sequentially Cohen-Macaulay} if there is a finite filtration of graded $S$-modules $0=M_0 \subset M_1 \subset \cdots \subset M_t=M$ such that all $M_i/M_{i-1}$ are Cohen-Macaulay, and the Krull dimensions of their quotients satisfy 
			$$ \dim(M_1/M_0)<\dim(M_2/M_1)<\cdots<\dim(M_t/M_{t-1}).$$
		\end{enumerate}	 
}\end{definition}
	A graph $G$ is said to be sequentially Cohen-Macaulay over $\mathbb{K}$ if $S/I(G)$ is sequentially Cohen-Macaulay.		 
\begin{definition}\label{component}
Let $I$ be a graded ideal of $S$. Then $I_{<j>}$ denotes the ideal generated by all degree $j$ elements of $I$. An ideal $I$ is called \emph{componentwise linear} if $I_{<j>}$ has a linear resolution for all $j$.
\end{definition} 
The following result of Herzog and Hibi  establishes a connection between the notions of componentwise linear ideals and sequentially Cohen-Macaulayness.
\begin{lemma}[Herzog and Hibi, \cite{HHibi}]\label{HH} Let $I$ be a squarefree monomial ideal of $S$. Then $S/I$ is sequentially Cohen-Macaulay if and only if $I^{\vee}$ is componentwise linear.
\end{lemma}

\begin{definition}
Let $I$ be a monomial ideal of $S$. Then $I$ is said to have \emph{linear quotients} if there is an ordering $u_1,\ldots,u_r$ of minimal generators of $I$ such that $\langle u_1,\ldots,u_{i-1}\rangle:\langle u_i\rangle $ is generated by a subset of $\{x_1,\ldots,x_n \}$ for all $i$.
\end{definition}
The following lemma shows that the concept of linear quotient is very useful to determine if an ideal has a linear resolution. 
\begin{lemma}\cite[Proposition 8.2.1]{Herzog'sBook}
	Let $I \subset S$ be a graded ideal generated in one degree and  $I$ has linear quotients. Then $I$ has a linear resolution.
\end{lemma}
\begin{definition}
	Let $I$ be a squarefree monomial ideal in $S$ with irredundant primary decomposition $I=\mathfrak{p}_1 \cap \cdots \cap \mathfrak{p}_r,$ where $\mathfrak{p}_i$ is an ideal generated by some variables in $S$. Then for $s \in \mathbb{N}_{>0},$ the $s$th \emph{symbolic power} of $I$, denoted by $I^{(s)}$, is defined as follows:
	$$I^{(s)}= \mathfrak{p}^s_1 \cap \cdots \cap \mathfrak{p}^s_r.$$
\end{definition}
The concept of polarization is a very useful tool to convert a monomial ideal into a squarefree  monomial ideal.
\begin{definition}
Let $u=\prod_{i=1}^{n}x_i^{m_i}$ be a monomial. Then the \textit{polarization} of $u$ in $T$ is the squarefree monomial $\widetilde{u}=\prod_{i=1}^{n}\prod_{j=1}^{m_i}x_{ij}$, where $T=\mathbb{K}[x_{11},x_{12},\ldots,x_{21},x_{22},\ldots,x_{n1},x_{n2},\ldots]$. If $I$ is a monomial ideal of $S$ generated by monomials $u_1,\ldots,u_m$, then the squarefree monomial ideal $\widetilde{I}=\langle \widetilde{u}_1,\dots, \widetilde{u}_m\rangle\subset T$ is called the polarization of $I$.
\end{definition}

\section{Powers of vertex cover ideals of trees}\label{sec3}
Fakhari \cite{Fakhari} introduces a new construction of graphs to obtain a Cohen-Macaulay very well covered graph from a given arbitrary graph $G$. Here we introduce a new construction of graphs to obtain a vertex decomposable graph from a given tree or unicyclic graph with cycle $C_n$, $n \neq 3,5$.\\
{\bf Notation}: The graph obtained from a given graph $G$ by deleting its isolated vertices is denoted by $G^{\circ}.$
	
{\bf Construction}:
Let $G$ be a simple graph with the vertex set $V=\{x_1,\ldots,x_n  \}$ and edges $E(G)=\{e_1,\ldots,e_t \}$. Let $p \in \mathbb{N}_{>0}$. Then for an edge $e=\{x_i,x_j \}$, we define a graph $e(p)$ with vertices $V(e(p))=\{x_{s,a}:s\in \{i,j \}, 1\leq a \leq p \}$ and edge set $E(e(p))=\{\{x_{i,l},x_{j,m}\}:l+m\leq p+1\}$. By convention, we set $e(0)$ to be an isolated  graph on vertices $V(e(0))=\{ x_{i,1},x_{j,1} \}$. Consider an ordered tuple $\mathbf{k}_t=(k_1,\ldots,k_t) \in \mathbb{N}^t$. 
Define a graph $G(\mathbf{k}_t)=G(k_1,\ldots,k_t)$ on new vertices
 $V(G(\mathbf{k}_t))=\cup_{i=1}^tV(e_i(k_i))$
  and the edge set $E(G(\mathbf{k}_t))= \cup_{i=1}^t E(e_i(k_i))$.

\begin{remark}\label{treeRem} Let $G$ be a tree on the vertex set $V(G)=\{x_1,\ldots,x_n \}$ and the edge set $E(G)=\{e_1,\ldots,e_{n-1} \}$.
	\begin{enumerate}[i)]
\item  Let $x_a \in V(G)$ be a vertex of degree $1$ and $x_b$ be the unique neighbour of $x_a$. Without loss of generality, we assume that $N_G(x_b)=\{x_a=x_{b_0},x_{b_1},\ldots,x_{b_r} \}$ with $\deg(x_{b_q})=1,$ for $0 \leq q \leq s$ and $s \leq r$. Let  $e_{c_q}=\{x_b,x_{b_q} \}$, where $q \in \{0,\ldots, r\}$. Further, assume that the edges incident with a vertex $x_{b_q}$ other than $e_{c_q}$ are $e_{i_{q-s-1}+1},\ldots,e_{i_{q-s}}$ for $q \in \{s+1,\ldots, r \}.$ 
	\item  Set $k=\max\{k_{c_0},\dots,k_{c_r} \}$. By deleting a vertex $x_{b,1}$ from $G(\mathbf{k}_{n-1})$ and identifying $x_{b,j}$ with $x_{b,j-1}$ for all $2 \leq j \leq k $, we get $$\left(G(\mathbf{k}_{n-1}) \setminus \{x_{b,1}  \}\right)^{\circ}  \simeq \left(G(k_1,\ldots,k'_{c_0},\ldots,k'_{c_r},\ldots,k_{n-1})\right)^\circ,$$ where $k'_{c_i}=\max \{0,k_{c_i}-1\}$.
	\end{enumerate}
We illustrate above remark with the help of following figures.
\end{remark}
\begin{minipage}{\linewidth}
	\begin{minipage}{.5\linewidth}
		\begin{figure}[H]
			\centering
			\includegraphics[scale=0.7]{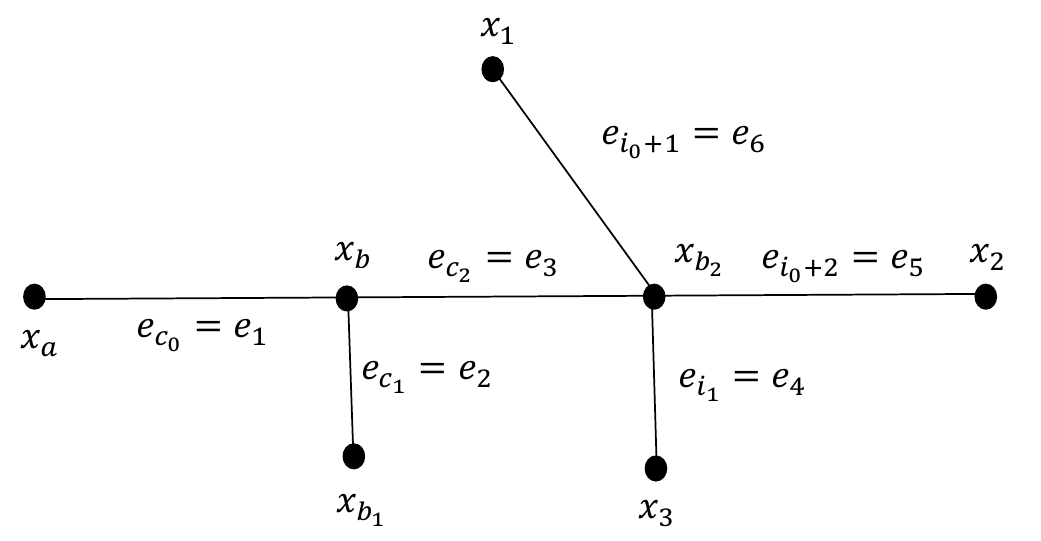}
			\caption*{$G$}
		\end{figure}
	\end{minipage}
\begin{minipage}{.4\linewidth}
	\begin{figure}[H]
		\centering
		\includegraphics[scale=0.55]{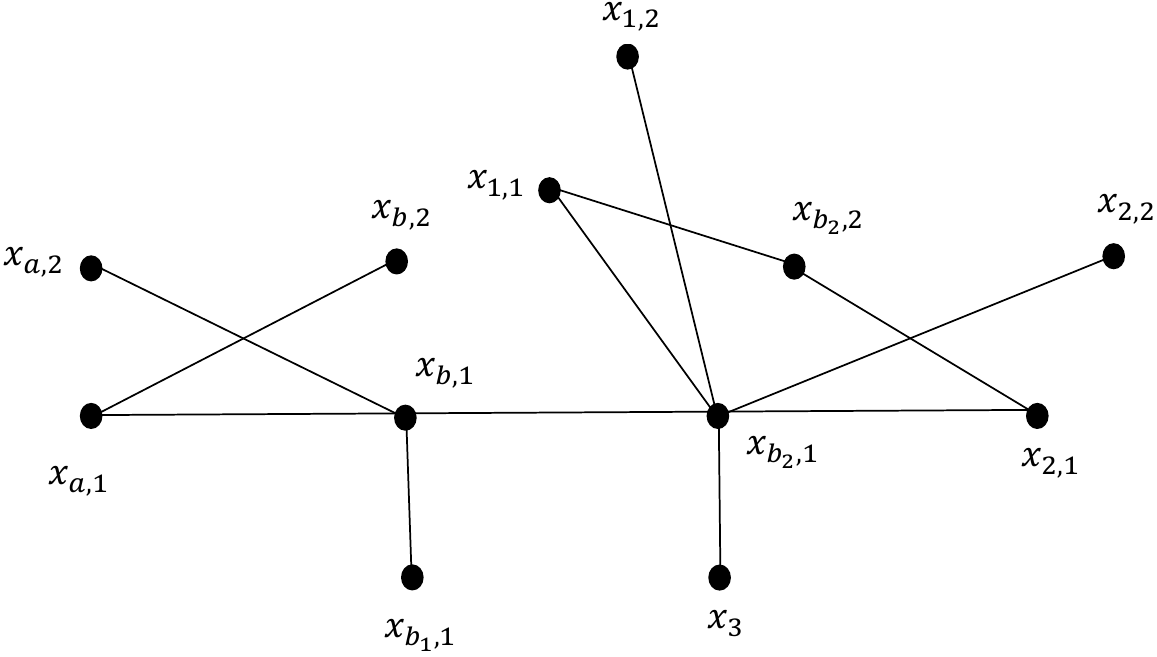}
		\caption*{$G(2,1,1,1,2,2)$}
	\end{figure}
\end{minipage}
\end{minipage}
\begin{minipage}{\linewidth}
	\begin{minipage}{.5\linewidth}
		\begin{figure}[H]
			\centering
			\includegraphics[scale=0.55]{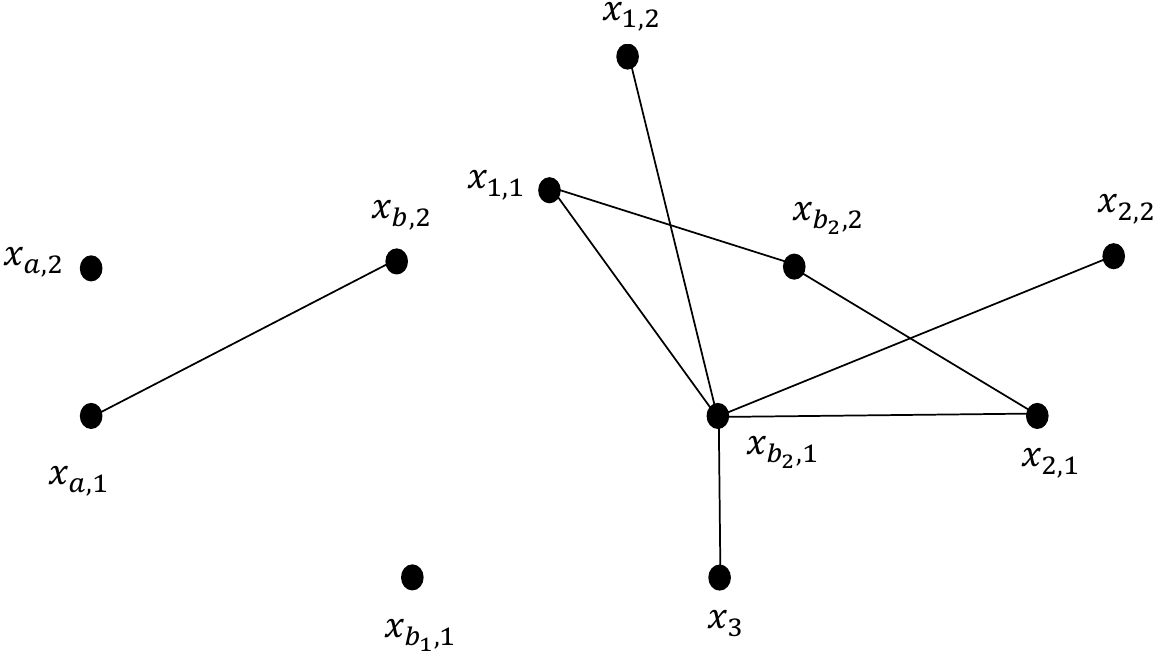}
			\caption*{$G(2,1,1,1,2,2)\setminus \{x_{b,1}\}$}
		\end{figure}
	\end{minipage}
	\begin{minipage}{.4\linewidth}
		\begin{figure}[H]
			\centering
			\includegraphics[scale=0.55]{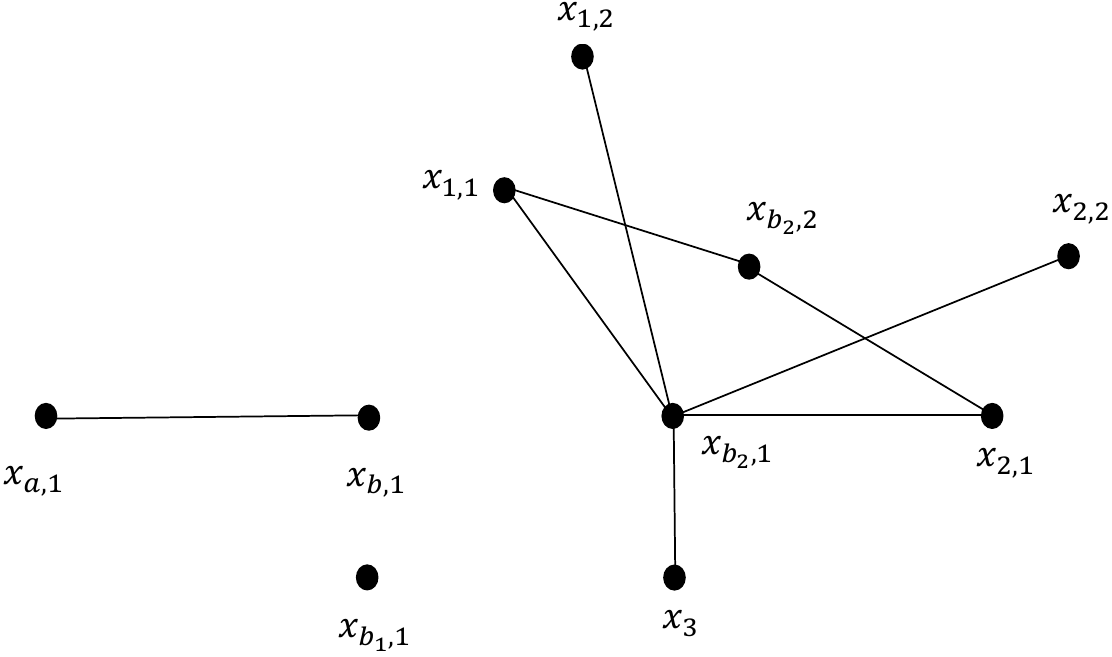}
			\caption*{$G(1,0,0,1,2,2)$}
		\end{figure}
	\end{minipage}
\end{minipage}

\begin{lemma}\label{neigh}
Let $G$ be a tree on the vertex set $V(G)=\{x_1,\ldots,x_n  \}$ and the edge set $E(G)=\{e_1,\ldots,e_{n-1}  \}$ with notations as in Remark \ref{treeRem}. Then $\left(G(\mathbf{k}_{n-1}) \setminus N_{G(\mathbf{k}_{n-1})}[x_{b,1}]\right)^\circ$ is isomorphic to $$\left( G \setminus N_G[x_b] \right)^{\circ}(\widehat{\mathbf{k}}_{n-1})\cup \bigcup\limits_{\substack{q\in \{s+1,\ldots,r\}, \\ j_q \in \{i_{q-s-1}+1,\ldots,i_{q-s}\}, \\ k_{j_q}-k_{c_q}>0}} e_{j_q}(k_{j_q}-k_{c_q}),$$ where $\widehat{\mathbf{k}}_{n-1}$ is obtained from ${\bf k}_{n-1}$ by deleting components $k_{j}'s$ corresponding to deleted edges $e_j's$, where $j \in \{c_1,\ldots,c_r\}\cup \bigcup\limits_{q=s+1}^{r}\{i_{q-s-1}+1,\ldots,i_{q-s}\}.$ 

\end{lemma}
\begin{proof}

 For  simplicity we denote  $G'=\left(G(\mathbf{k}_{n-1}) \setminus N_{G(\mathbf{k}_{n-1})}[x_{b,1}]\right)^\circ, G_1=\left( G \setminus N_G[x_b] \right)^{\circ}(\widehat{\mathbf{k}}_{n-1}),$\\
 $ G_2=  \bigcup\limits_{q=s+1}^r \bigcup\limits_{j_q=i_{q-s-1}+1}^{i_{q-s}} e_{j_q}(k_{j_q}-k_{c_q}),$ and $G''=G_1 \cup G_2.$ 
Observe that $N_{G(\mathbf{k}_{n-1})}[x_{b,1}]=\bigcup\limits_{q=0}^r\{x_{b_q,l}:1 \leq l \leq k_{c_q}  \} \cup \{x_{b,1}\}$. 
For $p \in [n]$ such that $x_p \in \left(   G \setminus N_G[x_b] \right)^{\circ},$ let $e_{a_{p-1}+1},\ldots,e_{a_p}$ be edges incident on $x_p$ in $G\setminus N_G[x_b]$. Then we define $q_p=\max \{k_{a_{p-1}+1},\ldots,k_{a_p}\}.$ Note that
$V(G'')=V_1 \sqcup V_2,$ where $V_1=\{x_{p,q}\in G(\mathbf{k}_{n-1}):x_p \in  \left(   G \setminus N_G[x_b] \right)^{\circ}, q \leq q_p\}$ and $V_2=V(G_2).$ Now define a map $f:V(G'') \rightarrow V(G')$ by $f(x)=x$ if $x \in V_1$ and if $x\in V_2\setminus V_1$, then $x=x_{b_q,t}$ for some $t$ and $f(x)=x_{b_q,k_{c_q}+t}$. Then $f$ is an isomorphism between $G''$ and $G'$.
\end{proof}
\begin{theorem}\label{tree}
Let $G$ be a tree on the vertex set $V(G)=\{x_1,\ldots,x_n \}$ and the edge set $E(G)=\{e_1,\ldots,e_{n-1}  \}$ with notations as in Remark \ref{treeRem}. Then $G(\mathbf{k}_{n-1})$ is a vertex decomposable graph, where $\mathbf{k}_{n-1} \in \mathbb{N}^{n-1}$.
\end{theorem}
\begin{proof}
We may assume that $n>1$. Now we use induction on $N=\sum\limits_{i=1}^{n-1}k_i $. If $N=0$, then $G(\mathbf{k}_{n-1})$ is a collection of isolated vertices, and hence vertex decomposable. We therefore suppose that  $N>0$. Observe that $N_{G(\mathbf{k}_{n-1})}[x_{a,k_{c_0}}] \subset N_{G(\mathbf{k}_{n-1})}[x_{b,1}]$. Thus,  by \cite[Lemma 4.2]{DE}, it is enough to show that $G(\mathbf{k}_{n-1})  \setminus \{x_{b,1}  \}$ and $G(\mathbf{k}_{n-1}) \setminus N_{G(\mathbf{k}_{n-1})}[x_{b,1}]$ are vertex decomposable. By Lemma \ref{neigh} and induction on $N$, it follows that  $G(\mathbf{k}_{n-1}) \setminus N_{G(\mathbf{k}_{n-1})}[x_{b,1}]$ is vertex decomposable. The fact that  $G(\mathbf{k}_{n-1})  \setminus \{x_{b,1}  \}$  is vertex decomposable follows from Remark \ref{treeRem} and induction on $N$.
\end{proof}
We denote $G(\mathbf{k})=G(k_1,\ldots,k_t)$ with $k_i=k$ for all $i$. 
Using the concept of polarizaton of a monomial ideal, Fakhari in \cite{Fakhari} showed that $J(G(\mathbf{k}))=\widetilde{J(G)^{(k)}}$, where $k \in \mathbb{N}_{>0}$. In the following, we prove Conjecture \ref{conjecture} for trees.

\begin{corollary}\label{treeCor}
	Let $G$ be a tree. Then for every $k \in \mathbb{N}_{>0}$, $J(G)^{k}$ has linear quotients, and hence it is componentwise linear. 
\end{corollary}
\begin{proof}
	In view of Theorem \ref{tree}, $G(\mathbf{k})$ is vertex decomposable, and hence shellable. Now using \cite[Theorem 8.2.5]{Herzog'sBook}, $J(G(\mathbf{k}))=\widetilde{J(G)^{(k)}}$ has linear quotients. By \cite[Lemma 3.5]{Fakhari}, $J(G)^{(k)}$ has a linear quotient, and hence by \cite[Theorem 8.2.15]{Herzog'sBook} componentwise linear. Now, the result follows from \cite[Corrolary 2.6]{GRV2005}.
\end{proof}
For ${\bf k}_t\in \mathbb{N}^{t}$, we have proved that $G({\bf k}_t)$ is vertex decomposable if $G$ is forest. Thus, it is natural to ask what are other classes of graphs for which $G({\bf k}_t)$ is vertex decomposable. In the following examples, we show that this is not true for bipartite vertex decomposable and chordal graphs.

\begin{example}\label{example}
Let $G$ be a graph on the vertex set $V(G)=\{x_1,x_2,x_3,x_4,x_5 \}$ and the edge set $E(G)=\{e_1=\{x_1,x_5\},e_2=\{x_1,x_2\}, e_3=\{x_2,x_3\},e_4=\{x_3,x_4\},e_5=\{x_1,x_4\} \}$ as shown in the figure. Then the fact $N_G[x_5] \subset N_G[x_1]$ implies that $x_1$ is a shedding vertex of $G$. Also, it is easy to see that $G \setminus \{x_1\}$ and $G \setminus N_G[x_1]$ are both vertex decomposable, and hence $G$ is a vertex decomposable graph. But $G(1,2,1,1,2)$ is not a vertex decomposable graph. To verify this note that $x_{11}$ is the unique shedding vertex of $G(1,2,1,1,2)$. However, $G(1,2,1,1,2) \setminus \{ x_{11}\}$ is a cycle $C_4$ which is not vertex decomposable.
\[
\begin{tikzpicture}[scale=1.7]
\draw (0,0)-- (0,1);
\draw (1,1)-- (0,1);
\draw (0,1)-- (0,0);
\draw (1,0)-- (1,1);
\draw (2,1)-- (1,1);
\draw (0,0)-- (1,0);
\begin{scriptsize}
\fill  (0,0) circle (1.5pt);
\draw[below] (0,0) node {$x_3$};
\fill  (0,1) circle (1.5pt);
\draw[above] (0,1) node {$x_2$};
\fill  (1,1) circle (1.5pt);
\draw[above] (1,1) node {$x_1$};
\fill  (1,0) circle (1.5pt);
\draw[below] (1,0) node {$x_4$};
\fill  (2,1) circle (1.5pt);
\draw[above] (2,1) node {$x_5$};
\end{scriptsize}
\draw[below] (1,-.2) node {$G$};
\draw[below] (5,-.2) node {$G(1,2,1,1,2)$};
\draw (4,1)-- (5,1);
\draw (5,0)-- (5,1);
\draw (4,0)-- (5,0);
\draw (4,1)-- (4,0);
\draw (5,1)-- (6,1);
\draw (4,2)-- (5,1);
\draw (4.48,0.5)-- (5,0);
\draw (4,1)-- (4.48,0.5);
\draw (6,0)-- (5,1);
\begin{scriptsize}
\fill  (4,1) circle (1.5pt);
\draw[above] (4,1) node {$x_{2,1}$};
\fill  (5,1) circle (1.5pt);
\draw[above] (5,1) node {$x_{1,1}$};
\fill (5,0) circle (1.5pt);
\draw[right] (5,0) node {$x_{4,1}$};
\fill (4,0) circle (1.5pt);
\draw[below] (4,0) node {$x_{3,1}$};
\fill (6,1) circle (1.5pt);
\draw[above] (6,1) node {$x_{5,1}$};
\fill (4,2) circle (1.5pt);
\draw[above] (4,2) node {$x_{2,2}$};
\fill (4.48,0.5) circle (1.5pt);
\draw (4.5,0.66) node {$x_{1,2}$};
\fill (6,0) circle (1.5pt);
\draw[right] (6,0) node {$x_{4,2}$};
\end{scriptsize}
\end{tikzpicture}
\]
\end{example}

\begin{example}\label{example2}
Let $G$ be a graph on the vertex set $V(G)=\{x_1,x_2,x_3,x_4,x_5 \}$ and the edge set 
$E(G)=\{e_1=\{x_1,x_5\},e_2=\{x_1,x_2\}, e_3=\{x_2,x_3\},e_4=\{x_3,x_4\},e_5=\{x_1,x_4\}, e_6=\{x_1,x_3 \} \}.$ Then $G$ is a chordal graph. Proceeding as in the above example one can check that $G(1,2,1,1,2,1)$ is not a vertex decomposable graph. 

\[
\begin{tikzpicture}[scale=1.7]
\draw (0,0)-- (0,1);
\draw (1,1)-- (0,1);
\draw (0,1)-- (0,0);
\draw (1,0)-- (1,1);
\draw (2,1)-- (1,1);
\draw (1,1)-- (0,0);
\draw (0,0)-- (1,0);
\begin{scriptsize}
\fill  (0,0) circle (1.5pt);
\draw[below] (0,0) node {$x_3$};
\fill  (0,1) circle (1.5pt);
\draw[above] (0,1) node {$x_2$};
\fill  (1,1) circle (1.5pt);
\draw[above] (1,1) node {$x_1$};
\fill  (1,0) circle (1.5pt);
\draw[below] (1,0) node {$x_4$};
\fill  (2,1) circle (1.5pt);
\draw[above] (2,1) node {$x_5$};
\end{scriptsize}
\draw[below] (1,-.2) node {$G$};
\draw[below] (5,-.2) node {$G(1,2,1,1,2,1)$};
\draw (4,1)-- (5,1);
\draw (5,0)-- (5,1);
\draw (4,0)-- (5,0);
\draw (4,1)-- (4,0);
\draw (5,1)-- (6,1);
\draw (4,2)-- (5,1);
\draw (6,0)-- (5,1);
\draw (5,1)-- (4,0);
\draw (4.28,0.52)-- (4,1);
\draw (4.28,0.52)-- (5,0);
\draw (4.28,0.52)-- (5,0);
\begin{scriptsize}
\fill  (4,1) circle (1.5pt);
\draw[above] (4,1) node {$x_{2,1}$};
\fill  (5,1) circle (1.5pt);
\draw[above] (5,1) node {$x_{1,1}$};
\fill (5,0) circle (1.5pt);
\draw[right] (5,0) node {$x_{4,1}$};
\fill (4,0) circle (1.5pt);
\draw[below] (4,0) node {$x_{3,1}$};
\fill (6,1) circle (1.5pt);
\draw[above] (6,1) node {$x_{5,1}$};
\fill (4,2) circle (1.5pt);
\draw[above] (4,2) node {$x_{2,2}$};
\fill (6,0) circle (1.5pt);
\draw[right] (6,0) node {$x_{4,2}$};
\fill (4.28,0.52) circle (1.5pt);
\draw[above](4.38,0.52) node {$x_{1,2}$};
\end{scriptsize}
\end{tikzpicture}
\]
\end{example}

The following example illustrate the fact that if $G$ is a vertex decomposable graph, then $G(\mathbf{k})$ need not be vertex decomposable.
\begin{example}\label{example3}
Let $G$ be a graph on the vertex set $V(G)=\{x_1,x_2,x_3,x_4\}$ and the edge set 
$E(G)=\{e_1=\{x_1,x_2\}, e_2=\{x_2,x_3\},e_3=\{x_3,x_4\},e_4=\{x_1,x_4\}, e_5=\{x_1,x_3 \} \}.$ Then $G$ is a vertex decomposable graph but $G(\mathbf{2})$ is not a vertex decomposable graph. 

\[
\begin{tikzpicture}[scale=1.7]
	\draw (0,0)-- (0,1);
	\draw (1,1)-- (0,1);
	\draw (0,1)-- (0,0);
	\draw (1,0)-- (1,1);
	\draw (1,1)-- (0,0);
	\draw (0,0)-- (1,0);
	\begin{scriptsize}
		\fill  (0,0) circle (1.5pt);
		\draw[below] (0,0) node {$x_3$};
		\fill  (0,1) circle (1.5pt);
		\draw[above] (0,1) node {$x_2$};
		\fill  (1,1) circle (1.5pt);
		\draw[above] (1,1) node {$x_1$};
		\fill  (1,0) circle (1.5pt);
		\draw[below] (1,0) node {$x_4$};
	\end{scriptsize}
	\draw[below] (1,-.2) node {$G$};
	\draw[below] (5,-.5) node {$G({\bf 2})$};
	\draw (4,1)-- (5,1);
	\draw (5,0)-- (5,1);
	\draw (4,0)-- (5,0);
	\draw (4,1)-- (4,0);
	\draw (5,1)-- (4,0);
	\draw (3.5,1.5)-- (5,1);
	\draw (3.5,1.5)-- (4,0);
	\draw (5.5,1.5)-- (4,1);
	\draw (5.5,1.5)-- (5,0);
	\draw (4,0)to[out=75, in=170] (5.5,1.5);
	\draw (3.5,-.5)-- (4,1);
	\draw (3.5,-.5)to[out=330, in=240] (5,0);
	\draw (5.5,-.5)-- (4,0);
	\draw (5.5,-.5)-- (5,1);
	\draw (5,1)to[out=255, in=10] (3.5,-.5);
	\begin{scriptsize}
		\fill  (4,1) circle (1.5pt);
		\draw[above] (4,1) node {$x_{2,1}$};
		\fill  (5,1) circle (1.5pt);
		\draw[above] (5,1) node {$x_{1,1}$};
		\fill (5,0) circle (1.5pt);
		\draw[below right ] (5,0) node {$x_{4,1}$};
		\fill (4,0) circle (1.5pt);
		\draw[below] (4,0) node {$x_{3,1}$};
		\fill  (3.5,1.5) circle (1.5pt);
		\fill  (5.5,1.5) circle (1.5pt);
		\fill  (5.5,-.5) circle (1.5pt);
		\fill  (3.5,-.5) circle (1.5pt);
		\draw[above] (5.5,1.5) node {$x_{1,2}$};
		\draw[above] (3.5,1.5) node {$x_{2,2}$};
		\draw[below] (3.5,-.5) node {$x_{3,2}$};
		\draw[below] (5.5,-.5) node {$x_{4,2}$};
	\end{scriptsize}
\end{tikzpicture}
\]

\end{example}
\section{Powers of vertex cover ideals of unicyclic graphs}\label{section4}

From Example \ref{example}, we observe that if $G$ is a unicyclic graph on $n$ vertices, then $G({\bf k}_n)$ need not be vertex decomposable.
On the other hand, if $G$ is a unicyclic vertex decomposable graph with cycle $C_n$, $n\neq 3,5$, then we show that  $G(\mathbf{k})$ is a vertex decomposable graph (see Theorem \ref{unicyclic} ). The following results will be useful in proving the main result of this section.
\begin{lemma}[Selvaraja, \cite{Selva2019}]\label{uni1}
Let $G$ be a graph and $\{x_1,\ldots,x_m \}\subset V(G)$. Set $\psi_0=G,$ $\psi_{i}=\psi_{i-1} \setminus \{x_i \}$, $\phi_i= \psi_{i-1} \setminus N_{\psi_{i-1}}[x_i]$ for all $1 \leq i \leq m $. Then $G$ is a vertex decomposable graph if it satisfies the following:
\begin{enumerate}[\rm i)]
\item $x_i$ is a shedding vertex of $\psi_{i-1}$ for all $1 \leq i \leq m$,
\item $\phi_i$ is vertex decomposable for all $1 \leq i \leq m$, and 
\item $\psi_m$ is vertex decomposable.
\end{enumerate}
\end{lemma}

In \cite{FKY}, Mohammadi, Kiani and Yassemi give a complete description of vertex decomposable unicyclic graphs which is noted in the following lemma.
\begin{lemma}[Mohammadi, Kiani and Yassemi, \cite{FKY}]\label{uni2}
Let $G$ be a unicyclic graph with cycle $C_n$, $n\neq 3,5$. Then $G$ is vertex decomposable if and only if  at least one whisker is attached to $C_n$. 
\end{lemma}
Now we prove the main theorem of this section.
\begin{theorem} \label{unicyclic}
Let $G$ be a unicyclic graph with cycle $C_n$, $n\neq 3,5$. If $G$ is vertex decomposable, then $G(\mathbf{k})$ is also vertex decomposable.
\end{theorem}
\begin{proof}
Using Lemma \ref{uni2}, there exists a whisker attached to $C_n$. Further, let $\{x_a,x_b \}$ be a whisker attached to $C_n$ at a vertex $x_b$. We set 
$$ \psi_0=G(\mathbf{k}), \psi_{i}=\psi_{i-1} \setminus \{x_{b,i} \}, \phi_i= \psi_{i-1} \setminus N_{\psi_{i-1}}[x_{b,i}]$$ for all $1 \leq i \leq k.$ 

In order to prove the theorem, first we show that $x_{b,i}$ is a shedding vertex in $\psi_{i-1}$ for $1\leq i\leq k$.
Since $x_b$ is the only vertex which is adjacent to $x_a$ in $G$, it follows from the definition of $G({\bf k})$ that a vertex $y$ is adjacent to $x_{a,k-i+1}$ if and only if $y=x_{b,j}$ for some $j$ with $k-i+1+j\leq k+1$. This implies that $N_{G({\bf k})}(x_{a,k-i+1})=\{x_{b,j}:1\leq j\leq i\}$. From the definition of $\psi_{i-1}=G({\bf k})\setminus\{x_{b,j}:1\leq j\leq i-1\}$, it follows that $x_{b,i}$ is the only vertex which is adjacent to ${x_{a,k-i+1}}$ in $\psi_{i-1}$.
Thus, using \cite[Lemma 4.2]{DE}, we get $x_{b,i}$ is a shedding vertex of $\psi_{i-1}$ for all $1 \leq i \leq k$.

Now, we show that $\phi_i$ is vertex decomposable for $1\leq i\leq k$.
From the definition of $\psi_{i-1}$, observe that $N_{\psi_{i-1}}(x_{b,j})\subset N_{\psi_{i-1}}(x_{b,i})$ for all $j\geq i$. This implies that, for $j>i$, $x_{b,j}$ is an isolated vertex in $\phi_i$. 
Set $H_1=G\setminus \{x_b\}$, $H_2=G\setminus N_G[x_b]$ and assume that $|E(H_1)|=r$. Note that $\phi_1$ is isomorphic to $H_2({\bf k})$ with some isolated vertices, and hence vertex decomposable by Theorem \ref{tree}. For $i\geq 2$, identify  $x_{l,j}$ with $x_{l,j-k-1+i}$ in $\phi_i$ for all $x_l\in N_G(x_b)$. Now, after deleting isolated vertices, we get $\phi_i$ is isomorphic to $H_1(k_{i_1},\dots, k_{i_r})$, where $k_{i_j}\leq k$. Since $H_1$ is forest, by Theorem \ref{tree}, we get that $\phi_i$ is vertex decomposable.

Now, using Lemma \ref{uni1}, it remains to show that $\psi_k$ is vertex decomposable. This follows from Theorem \ref{tree} and the fact that $\psi_k=H_1({\bf k})$, and $H_1$ is a forest.
\end{proof}
As an immediate consequence, we get the following result.  
\begin{corollary}\label{uniCor}
Let $G$ be a unicyclic vertex decomposable graph with cycle $C_n$, $n\neq 3,5$. Then all symbolic powers of $J(G)$ are componentwise linear.	
\end{corollary}
\begin{proof}
	For $k\geq 1$, by Theorem \ref{unicyclic}, $G({\bf k})$ is vertex decomposable, and hence sequentially Cohen-Macaulay. Using Lemma \ref{HH}, we get that $J(G({\bf k}))$ is componentwise linear. Since $\widetilde{J(G)^{(k)}}=J(G({\bf k}))$, $J(G)^{(k)}$ is componentwise linear. 
\end{proof}
For a graph $G$, we denote $\deg(J(G))$ is the maximum degree of minimal monomial generators of $J(G)$.
As an application of Corollary \ref{treeCor} and Corollary \ref{uniCor}, we have the following:
\begin{theorem}\hfill{}
	\begin{enumerate}[\rm i)]
		\item Let $G$ be a tree. Then $\reg(J(G)^s)=s\deg(J(G))$ for $s\geq 1$.
		\item Let $G$ be a unicyclic vertex decomposable graph with cycle $C_n$, $n\neq 3,5$. Then $\reg(J(G)^{(s)})=s\deg(J(G))$ for $s\geq 1$.
	\end{enumerate}
\end{theorem}

\renewcommand{\bibname}{References}
\renewcommand{\bibname}{References}
\bibliographystyle{plain}  
\bibliography{refs_reg}
\end{document}